\documentclass[10pt,twocolumn,a4paper]{article}
\usepackage[T1]{fontenc}
\usepackage[utf8]{inputenc}
\usepackage{lmodern,amsmath,amssymb,amsthm,mathtools,bm}
\usepackage{graphicx,microtype,booktabs,array}
\usepackage{epstopdf}
\graphicspath{{./}{SGMS_Theory_Draft6/}}
\usepackage[hidelinks]{hyperref}
\usepackage{geometry}
\allowdisplaybreaks
\newtheorem{theorem}{Theorem}

\newtheorem{proposition}[theorem]{Proposition}
\newtheorem{corollary}[theorem]{Corollary}
\theoremstyle{definition}\newtheorem{definition}[theorem]{Definition}
\newtheorem{example}[theorem]{Example}
\theoremstyle{remark}\newtheorem{remark}[theorem]{Remark}
\newcommand{\R}{\mathbb R}
\newcommand{\norm}[1]{\lVert #1\rVert}
\newcommand{\dd}{\,\mathrm d}
\newcommand{\G}{\mathcal G}
\newcommand{\W}{\mathcal W}

\newcommand{\K}{\mathcal K}

\newcommand{\op}{\mathrm{op}}

\title{\bfseries Geometric Stability of Dynamical Systems\\[2pt] with Self-Generated State-Space Memory}
\author{Robert Vrabel\\[-1pt]\small Faculty of Materials Science and Technology in Trnava\\[-2pt]\small Slovak University of Technology in Bratislava, Slovakia\\ {\small\tt robert.vrabel@stuba.sk}}
\date{}
\begin{document}
	\twocolumn[
	\maketitle
	
	\begin{abstract}
		We develop a geometric stability framework for dynamical systems coupled to memory fields generated by their own state trajectories. Temporal age and state-space displacement enter a nonlinear hereditary force admitting an exact, trajectory-dependent velocity representation. A common velocity-history phase space contains both resting and uniformly translating solutions. We establish global velocity arrest under a uniform geometric memory-gain bound and obtain exact translating-speed balance laws from the full radial imprint. The response slope yields a positive longitudinal exponent on rising-speed branches, while a complementary contraction criterion gives local exponential attraction to selected descending branches in the scalar velocity dynamics, in fading-memory topology. Our principal result is a constructive geometric non-identifiability theorem: two systems with identical nonlinear force operators on an open neighborhood of rest can exhibit sharply different complete-motion behavior at the same subcritical gain. In the first, every admissible trajectory with bounded velocity prehistory comes to rest exponentially; in the second, locally attracting rest coexists with a locally attracting translating branch and an additional spectrally unstable translating branch. The infinite-memory case exhibits exponentially flat force differences, with all finite-order equilibrium derivatives agreeing despite the possibility of distinct global translating families. The framework also yields classical Volterra kernels and distributed-derivative transfer functions by linearization.
	\end{abstract}
	
	\vspace{2mm}
	
	
	\par\vspace{4mm}
	
	
	\par\vspace{10mm}
	]
	
	\section{Introduction}
An ordinary time-memory law assigns a weight to a past state according to how long ago it occurred. A self-generated spatial memory adds a second criterion: \emph{where} that state was visited relative to the current position. A recent state far from the particle and an older state close to it can therefore contribute differently to the present force. The trajectory deposits a spatial record, then interacts with that very record as it moves. In mathematical terms, the temporal age $s$ and the spatial displacement $x(t)-x(t-s)$ enter the same nonlinear hereditary functional.

This coupling creates a distinction between two questions that are often conflated. Local analysis near a stationary state tests the response to small perturbations and depends on a limited collection of derivatives and temporal moments. Global analysis asks whether the same system can sustain a persistent motion, return to a previously visited region, or converge to rest from a distant history. The latter questions involve the full geometry of the imprints along nonzero displacements. Our central problem is to determine how much of this global behavior remains undetermined even when the \emph{complete local nonlinear memory force} is known exactly.

The starting point is the nonlinear field--trajectory coupling studied in memory-driven particle models such as \cite{Sarkar2025}, alongside earlier pilot-wave trajectory equations \cite{Oza2013} and trail-following active-matter models \cite{Lutier2026}. The present analysis develops a common force functional and identifies the geometric quantities controlling velocity arrest, uniform translations, branch orientation, and moving-orbit stability. The central dynamical construction distinguishes two systems whose forces agree throughout an open neighborhood of rest while their global arrest and moving-branch behaviors differ.

The contribution is organized around three levels of information. The \emph{equilibrium tangent dynamics} describes what infinitesimal perturbations detect. The \emph{exact traveling response} decides which constant nonzero velocities satisfy the full nonlinear equation. Finally, the \emph{moving-branch variation and nonlinear estimates} distinguish growing and locally attracting speed histories. These levels use successively more of the spatial imprint; none of the latter two can generally be reconstructed from equilibrium spectral data alone.

The principal structural theorem makes that gap explicit. We construct two bounded, smooth, strictly radially decreasing imprints, subject to the same positive temporal weights and the same feedback gain. Their \emph{entire nonlinear memory forces}, including all local derivatives, agree on an open ball in one common velocity-history space. Nevertheless, every admissible bounded-history trajectory of the first system undergoes exponential velocity arrest, whereas the second possesses an attracting rest state, a locally attracting nonzero translating speed, and an additional translating branch with a positive longitudinal characteristic exponent. Thus the result concerns an exact operator-level equality together with a rigorously established contrast in complete motions, not merely two different linearizations or an algebraic multiplicity example.

The mathematical context includes Volterra integral equations \cite{Gripenberg1990}, classical stability theory for nonlinear Volterra convolution equations \cite{Staffans1975}, positive-kernel spectral theory \cite{Naito2007,Ngoc2008}, infinite-delay phase spaces \cite{HaleKato1978}, and the classical fading-memory framework \cite{ColemanMizel1966,ColemanMizel1968}. The field--trajectory feedback is motivated by self-generated trail and pilot-wave models, including \cite{Sarkar2025,Oza2013,Lutier2026}. Existing hereditary theory supplies well-posedness and stability techniques; the present geometric construction identifies the precise nonlinear information that these equilibrium-level reductions do not retain. We first derive an exact geometry-to-velocity representation, then connect the complete traveling response to branch existence and stability. The resulting constructive separation theorem is followed by an infinite-memory flatness result and the corresponding classical control reductions.

For clarity, the assertions about moving branches are formulated at two distinct levels: exponential attraction is established for scalar speed histories with sufficiently small compatible initial perturbations, and a positive characteristic exponent is an assertion about the longitudinal variational equation. Neither assertion is used as a substitute for a global classification of all histories or for a description of basin boundaries. All results below concern the exact hereditary evolution.

\section{State histories and geometric memory}
For a continuously differentiable trajectory $x:\R\to\R^n$ let $x_t(\theta)=x(t+\theta)$, $\theta\le0$, and set
\begin{equation}\label{eq:increment}
 \Delta_sx(t)=x(t)-x(t-s),\qquad s\ge0.
\end{equation}
Given $m\in L^1(0,\infty)$, $m\ge0$, a spatial imprint $\Phi:\R^n\to\R$, and a nonnegative history-dependent weighting $\W(s;h)$, define
\begin{align}
 S_x(z,t)&=\int_0^\infty \W(s;x_t)\Phi(z-x(t-s))\dd s,\label{eq:field}\\
 \G[h]&=-\int_0^\infty\W(s;h)\nabla\Phi\bigl(h(0)-h(-s)\bigr)\dd s.\label{eq:geomforce}
\end{align}
Our basic dynamical law is
\begin{equation}\label{eq:main}
 \dot x(t)=\int_0^\infty m(s)\dot x(t-s)\dd s+\kappa\G[x_t],
 \quad\kappa>0.
\end{equation}
Notice that $\G[x_t]=-\nabla_zS_x(x(t),t)$: differentiation is in $z$, with the history and its weights held fixed. We call \eqref{eq:geomforce} a \emph{self-generated geometric hereditary force}. Its age selector is $s$; its geometric selector is $\Delta_sx(t)$.

\begin{definition}[Admissible histories]\label{def:history}
A past velocity $v_-\in C_b(( -\infty,0];\R^n)$ is admissible when a $C^1$ prehistory $x_-:(-\infty,0]\to\R^n$ satisfies $x_-'=v_-$ and all force integrals are finite. This broader class $C_b$ is used for existence and global-arrest statements. The uniformly continuous subclass $BUC(( -\infty,0];\R^n)$ is used below as the common Banach phase space for local operator differentiation and moving-branch stability. A \emph{classical solution} on $[0,T]$ is a continuous extension of $x_-$, differentiable for $t>0$, with continuous, bounded velocity on compact time intervals, satisfying \eqref{eq:main}. Compatibility at $t=0$ is assumed whenever continuity of the joined derivative is needed. Complete translating histories $x(t)=x_0+vt$ are considered separately and need not be bounded in position.
\end{definition}

For fixed $\W(s;h)=w(s)\ge0$, bounded $\Phi$ and $\nabla\Phi$, and $w\in L^1$, the field and force are defined even for unbounded translating histories. A convenient sufficient condition for the nonlinear factorization below is $\int_0^\infty s\W(s;x_t)\dd s<\infty$, uniform bounds on past velocities, and bounded spatial Hessian. For equilibrium differentiability we use the Banach space $X=BUC(( -\infty,0];\R^n)$ with its supremum norm. A concrete sufficient hypothesis is $\Phi\in C_b^2$ with uniformly continuous Hessian and $h\mapsto\W(\cdot;h)$ continuously Fr\'echet differentiable from $X$ into $L^1((0,\infty),(1+s)\dd s)$; then differentiation of \eqref{eq:geomforce} is justified. Moving reference histories are treated on affine translates $h_v+X$, where $h_v(-s)=x_0-vs$, and are not required to be bounded. For infinite-memory stability rates, exponential moments are stated explicitly.

\begin{remark}[Translations and signs]\label{rem:sign}
If the weight obeys $\W(s;h+c)=\W(s;h)$, the force is invariant under $x\mapsto x+c$. Accordingly, equilibrium points are not isolated. For a strictly decreasing radial imprint, $\nabla\Phi(\Delta_sx)$ points toward the historical state, while $-\kappa\nabla\Phi(\Delta_sx)$ points away from that state. The latter is repulsive along a straight track, but its instantaneous power can have either sign along curved or returning trajectories. Thus a universal ``memory-friction'' description is incorrect.
\end{remark}

\section{A common phase space for resting and moving states}\label{sec:phase}
A complete translating path $x(t)=x_0+vt$ is unbounded in position, while its velocity is constant. Thus an unweighted bounded-position history space cannot contain all the steady motions of interest. Passing to velocity histories removes this obstacle: rest corresponds to the zero function and a uniformly translating solution to a nonzero constant function, both on the same phase space. The position $x_0$ is retained as an additive coordinate.

Classical fading-memory spaces, developed in particular by Coleman and Mizel \cite{ColemanMizel1966,ColemanMizel1968}, allow histories to be measured with decreasing weights assigned to increasingly remote times. Here we first impose an unweighted bound on \emph{initial} velocity perturbations. This bound gives immediate control of every displacement integral needed in the nonlinear contraction arguments. Weighted fading-memory norms are then used to describe convergence of the \emph{evolved} histories. Starting directly with smallness in a fading norm would permit large perturbations in the distant past and would require additional continuity assumptions on the corresponding evolution operator. The phase space below is the uniformly continuous subclass of the broader admissible class introduced in Definition~\ref{def:history}; all local stability statements are made for histories in this Banach space. Set
\[
 \mathcal V=BUC(( -\infty,0];\R^n),\quad
 (\mathsf d_s v)=\int_0^s v(-u)\dd u.
\]
For fixed $w\ge0$ define
\begin{align}
 \mathfrak F_\Phi(v)&=-\int_0^\infty w(s)
     \nabla\Phi(\mathsf d_s v)\dd s,\label{eq:velocityforce}\\
 \mathfrak T_\Phi(v)&=\int_0^\infty m(s)v(-s)\dd s
        +\kappa\mathfrak F_\Phi(v).\label{eq:velocityoperator}
\end{align}
Every classical orbit with velocity $v(t)=\dot x(t)$ satisfies
\begin{equation}\label{eq:abstract_velocity}
 v(t)=\mathfrak T_\Phi(v_t),\qquad
 x(t)=x(0)+\int_0^t v(u)\dd u.
\end{equation}
Conversely, a continuous global solution of the first equation reconstructs a trajectory solving \eqref{eq:main}. The space $\mathcal V$ simultaneously contains the constant zero history (rest) and every constant nonzero history $v(\theta)\equiv v_0$ (uniform translation). Translating orbits become fixed histories in the reduced phase variable.

\begin{proposition}[Continuity and geometric Lipschitz bound]\label{prop:commonphase}
Suppose $m,w\in L^1(0,\infty)$ are nonnegative, $M_1=\int_0^\infty sw(s)\dd s<\infty$, $\nabla\Phi(0)=0$, and $\sup_z\norm{\nabla^2\Phi(z)}_{\op}\le L$. Then $\mathfrak T_\Phi:\mathcal V\to\R^n$ is globally Lipschitz and
\begin{equation}\label{eq:commonlip}
 \norm{\mathfrak T_\Phi(v)-\mathfrak T_\Phi(u)}
 \le(M_0+\kappa LM_1)\norm{v-u}_\infty.
\end{equation}
If $\Phi$ is radial and $\mathfrak F_\Phi$ uses $f(z)=a(\norm z)z$, the derivative at a constant history $v_0=re$ is
\begin{equation}\label{eq:derivmoving}
 D\mathfrak F_\Phi(v_0)z
 =\int_0^\infty w(s)Df(sv_0)
  \int_0^s z(-u)\dd u\dd s,
\end{equation}
provided $Df$ is bounded and continuous.
\end{proposition}
\begin{proof}
For every pair of histories, $\norm{\mathsf d_sv-\mathsf d_su}\le s\norm{v-u}_\infty$. The mean-value inequality for $\nabla\Phi$ and integration against $m,w$ yield \eqref{eq:commonlip}. For \eqref{eq:derivmoving}, differentiate the Nemytskii factor $f(\mathsf d_sv)$ and dominate the difference quotients by $s\norm{Df}_\infty w(s)\norm z_\infty$. The resulting integral converges by the first-moment hypothesis. Continuity of the derivative follows by dominated convergence.
\end{proof}

The position-based and velocity-based representations encode different local neighborhoods. A uniform bound on velocity increments controls $\mathsf d_s v$ by $s\norm v_\infty$, which grows with the memory age. This distinction becomes decisive for exponentially decaying infinite-memory kernels and is quantified in Theorem~\ref{thm:flat}.

\section{Exact nonlinear factorization}
The first step identifies where geometric structure enters without making a small-displacement approximation. A classical convolution kernel associates a fixed influence with every time lag. Here the spatial displacement accumulated during that lag changes the direction and magnitude of the feedback. The following exact identity rewrites this effect as a history-dependent velocity kernel. This provides the main bridge between the full spatial picture and hereditary velocity equations without suppressing the geometry. Define the symmetric matrix field
\begin{equation}\label{eq:matrixA}
 \mathsf A(z)=-\int_0^1\nabla^2\Phi(\theta z)\dd\theta.
\end{equation}
As $\nabla\Phi(0)=0$, the fundamental theorem of calculus gives the exact identity $-\nabla\Phi(z)=\mathsf A(z)z$.

\begin{theorem}[Exact geometry-to-velocity reduction]\label{thm:factor}
Suppose $\Phi\in C^2(\R^n)$, $\nabla\Phi(0)=0$, $\sup_z\norm{\mathsf A(z)}_{\op}<\infty$, and a classical solution has bounded past velocity with
\begin{equation}\label{eq:fubini_assump}
 \int_0^\infty s\W(s;x_t)\dd s<\infty.
\end{equation}
Then \eqref{eq:main} is equivalent along this solution to
\begin{align}
 \dot x(t)&=\int_0^\infty\K_x(t,r)\dot x(t-r)\dd r,\label{eq:velrep}\\[-2pt]
 \K_x(t,r)&=m(r)I+\kappa\int_r^\infty
  \W(s;x_t)\mathsf A(\Delta_sx(t))\dd s.
 \label{eq:kernel}
\end{align}
The effective kernel generally depends on the entire trajectory and is not a prescribed convolution kernel.
\end{theorem}
\begin{proof}
From $-\nabla\Phi(z)=\mathsf A(z)z$ and absolute continuity,
\[
 \kappa\G[x_t]=\kappa\int_0^\infty\W(s;x_t)
 \mathsf A(\Delta_sx(t))\int_0^s\dot x(t-r)\dd r\dd s.
\]
If $\norm{\dot x(t-r)}\le V$ on the relevant past, the norm of the double integrand is bounded, after integration, by $\kappa V\norm{\mathsf A}_{\infty}\int s\W(s;x_t)\dd s$. This is finite by \eqref{eq:fubini_assump}; Tonelli's theorem for the scalar majorant and Fubini's theorem for the vector integrand apply. Exchanging $r$ and $s$ yields \eqref{eq:kernel}, exactly and without a Taylor remainder.
\end{proof}

For a radial $\Phi(z)=\phi(\norm z)$ with
\begin{equation}\label{eq:radiala}
 -\nabla\Phi(z)=a(\norm z)z,\qquad a(r)\ge0,
\end{equation}
the same identity admits a particularly transparent \emph{alternative factorization} (not necessarily the matrix \eqref{eq:matrixA}):
\begin{equation}\label{eq:radialkernel}
 \K_x(t,r)=\left[m(r)+\kappa\int_r^\infty
 \W(s;x_t)a(\norm{\Delta_sx(t)})\dd s\right]I.
\end{equation}
Spatial return distance and temporal age are inseparable along a generic trajectory. Even with a fixed temporal weight $w(s)$, the effective coefficient in \eqref{eq:radialkernel} changes as the particle visits different regions of its own trail. Thus the time kernel inside the exact velocity equation is generated by the evolving history, although the underlying temporal decay law is prescribed. This is the geometric information that will be lost when the full operator is replaced by its equilibrium tangent.

The instantaneous memory power is, exactly,
\begin{equation}\label{eq:power}
 \kappa\dot x(t)^\top\G[x_t]=\kappa\int_0^\infty
 \W(s;x_t)a(\norm{\Delta_sx})\Delta_sx^\top\dot x(t)\dd s.
\end{equation}
For a straight track every integrand is nonnegative; after a turn, $\Delta_sx^\top\dot x(t)$ can be negative. Thus the same positive radial imprint can inject or remove energy depending on trajectory geometry.

\section{Well-posedness and global velocity arrest}
The first global stability question is whether a trajectory can eventually lose all speed even while it retains a spatial record of its past. We call this \emph{velocity arrest}: $\dot x(t)$ decays and the position approaches a history-dependent limit $x_\infty$. Translation invariance precludes choosing one distinguished destination in advance. The next results establish a solution class and a sufficient uniform arrest criterion; later we shall construct a second imprint that violates this global conclusion while sharing the same complete local force at rest.

\begin{proposition}[A well-posed subcritical fixed-kernel class]\label{prop:exist}
Let $\W(s;h)=w(s)\ge0$, $M_0=\int_0^\infty m(s)\dd s<\infty$, $M_1=\int_0^\infty sw(s)\dd s<\infty$, and suppose $\nabla\Phi(0)=0$ and $\nabla\Phi$ is globally Lipschitz with constant $L$. If
\begin{equation}\label{eq:existcrit}
 M_0+\kappa L M_1<1,
\end{equation}
then every bounded continuous past velocity gives a unique bounded future velocity solving \eqref{eq:main}, with the past position fixing the additive translation. The joined trajectory is $C^1$ at zero if the initial velocity is compatible with the equation.
\end{proposition}
\begin{proof}
Let $v$ be an unknown bounded continuous future velocity and extend it by the prescribed past velocity $v_-$ to $\widetilde v$ on $(-\infty,\infty)$. The displacement entering the force is exactly $\int_0^s\widetilde v(t-r)\dd r$. Hence the right-hand side of \eqref{eq:main} defines the map
\begin{align*}
 (Tv)(t)={}&\int_0^\infty m(s)\widetilde v(t-s)\dd s\\
 &-\kappa\int_0^\infty w(s)\nabla\Phi\left(
 \int_0^s\widetilde v(t-r)\dd r\right)\dd s.
\end{align*}
Since $\nabla\Phi(0)=0$, the estimate $\norm{\nabla\Phi(z)}\le L\norm z$ makes $T$ a self-map of $C_b([0,\infty);\R^n)$. For two future velocities $v_1,v_2$ with the same prehistory, all differences in the integrals are bounded by $\norm{v_1-v_2}_\infty$, giving
\[
 \norm{Tv_1-Tv_2}_\infty\le
 (M_0+\kappa LM_1)\norm{v_1-v_2}_\infty.
\]
Dominated convergence ensures continuity in $t$; the inequality \eqref{eq:existcrit} makes $T$ a contraction. Its unique fixed point is the required future velocity. Integrating it gives the trajectory and the indicated compatibility statement.
\end{proof}

This Banach fixed-point construction provides an explicit global solution class. History-adaptive weights are handled by the uniform estimates of Theorem~\ref{thm:arrest} along their admissible classical solutions.

\begin{theorem}[Uniform history-adaptive arrest]\label{thm:arrest}
Let $x$ be a global classical solution with $V_-=\sup_{t\le0}\norm{\dot x(t)}<\infty$. Let $L_A=\sup_z\norm{\mathsf A(z)}_{\op}<\infty$. Suppose that for some $\gamma>0$
\begin{align}
 Q_\gamma:={}&\int_0^\infty e^{\gamma s}m(s)\dd s\\
 &+\kappa L_A\sup_{t\ge0}\int_0^\infty
 \frac{e^{\gamma s}-1}{\gamma}\W(s;x_t)\dd s<1.
 \label{eq:Qgamma}
\end{align}
Then there exists $x_\infty\in\R^n$ such that
\begin{equation}\label{eq:arrest}
 \norm{\dot x(t)}\le V_-e^{-\gamma t},\qquad
 \norm{x(t)-x_\infty}\le \frac{V_-}{\gamma}e^{-\gamma t}.
\end{equation}
If \eqref{eq:Qgamma} holds uniformly over admissible histories, this conclusion holds for every admissible global solution.
\end{theorem}
\begin{proof}
Theorem~\ref{thm:factor} and its scalar majorant imply
\begin{align*}
\norm{\dot x(t)}\le {}&\int_0^\infty m(s)\norm{\dot x(t-s)}\dd s\\
 &+\kappa L_A\int_0^\infty\W(s;x_t)
  \int_0^s\norm{\dot x(t-r)}\dd r\dd s.
\end{align*}
Fix $T>0$, and let
\[
 B_T=\max\bigl\{V_-,\sup_{0\le t\le T}e^{\gamma t}\norm{\dot x(t)}\bigr\}.
\]
The bound $\norm{\dot x(u)}\le B_Te^{-\gamma u}$ holds for all $u\le T$: for $u\le0$ the right side is at least $V_-$, while on $[0,T]$ it follows by definition. Multiplication by $e^{\gamma t}$ gives
\begin{align*}
 e^{\gamma t}\norm{\dot x(t)}\le B_T\biggl[&
 \int e^{\gamma s}m(s)\dd s\\
 &+\kappa L_A\int\frac{e^{\gamma s}-1}{\gamma}
 \W(s;x_t)\dd s\biggr].
\end{align*}
The bracket is at most $Q_\gamma$, by hypothesis. If the maximum defining $B_T$ were attained above $V_-$, continuity would yield $B_T\le Q_\gamma B_T<B_T$, a contradiction. Hence $B_T\le V_-$. Sending $T\to\infty$ proves the first bound. Finally $\int_t^\infty\norm{\dot x(u)}\dd u\le V_-e^{-\gamma t}/\gamma$ gives existence of $x_\infty$ and the second estimate.
\end{proof}

\begin{remark}[Relation to small gain]
Condition \eqref{eq:Qgamma} is a trajectory-uniform weighted small-gain inequality. The exact factorization also admits sharpened trajectory-dependent estimates through the radial return profile $a(\norm{\Delta_sx})$ and its anisotropic extensions.
\end{remark}

\begin{corollary}[Convergence of spatial memory]\label{cor:fieldconv}
Assume in addition $\W=w\in L^1(0,\infty)$ fixed and $\Phi,\nabla\Phi\in BUC(\R^n)$. Then
\begin{align}
 \norm{S_x(\cdot,t)-W_0\Phi(\cdot-x_\infty)}_\infty&\to0,\label{eq:Sconv}\\
 \norm{\nabla S_x(\cdot,t)-W_0\nabla\Phi(\cdot-x_\infty)}_\infty&\to0,
\end{align}
where $W_0=\int_0^\infty w(s)\dd s$.
\end{corollary}
\begin{proof}
For each fixed $s\ge0$, the argument $t-s$ tends to infinity, so $x(t-s)\to x_\infty$. Uniform continuity of $\Phi$ implies that the spatial supremum of the difference between the two translated imprints tends to zero. The difference is bounded by $2\norm\Phi_\infty$. Dominated convergence in $s$ with weight $w(s)$ proves the first assertion. The same argument applied to $\nabla\Phi$ proves the second.
\end{proof}

For adaptive $\W$, convergence of the state alone is insufficient: the weights themselves must converge (for example in $L^1$) to recover an asymptotic field. The theorem therefore distinguishes arrest of the trajectory from convergence of the jointly evolving memory state.

\section{Tangent theory and its geometric blind spots}
A reader familiar with linear control may ask whether the complete force can be summarized by one transfer function. Near rest this is indeed possible: the spatial imprint contributes its Hessian at the origin and the temporal weight contributes its Laplace transform. The reduction correctly describes the equilibrium tangent problem. Yet it cannot detect the forces exerted after a finite displacement from earlier states. The following calculation makes the scope of the linear description precise.

Let $h_\ast(\theta)\equiv x_\ast$ be a constant history. Suppose the weight map $h\mapsto\W(\cdot;h)$ is Fr\'echet differentiable in a weighted $L^1$ space around $h_\ast$, with domination sufficient for the product rule in \eqref{eq:geomforce}. Write $w_\ast(s)=\W(s;h_\ast)$.

\begin{proposition}[Equilibrium blindness to weight adaptation]\label{prop:blind}
If $\Phi\in C^2$ and $\nabla\Phi(0)=0$, then
\begin{equation}\label{eq:blind}
 D\G[h_\ast]y=-\nabla^2\Phi(0)
 \int_0^\infty w_\ast(s)[y(0)-y(-s)]\dd s.
\end{equation}
In particular, $D_h\W(\cdot;h_\ast)$ is absent. Different adaptive weights with the same values at constant histories have identical equilibrium force linearizations.
\end{proposition}
\begin{proof}
Differentiate the integrand $\W(s;h)\nabla\Phi(h(0)-h(-s))$ at $h_\ast$. The product-rule contribution containing $D_h\W$ is multiplied by $\nabla\Phi(0)=0$, and hence vanishes. The remaining term is the displayed Hessian action. Weighted domination permits passage of the derivative through the integral.
\end{proof}

Away from equilibrium, the missing information reappears explicitly. The following identity makes the contrast precise.

\begin{proposition}[Full history-dependent first variation]\label{prop:vari}
At a history $h$ where differentiation is valid, put $d_s=h(0)-h(-s)$ and $\delta_sy=y(0)-y(-s)$. Then
\begin{align*}
 D\G[h]y={}&-\int_0^\infty\W(s;h)
  \nabla^2\Phi(d_s)\delta_sy\dd s\label{eq:fullvar}\\[-2pt]
 &-\int_0^\infty D_h\W(s;h)[y]\nabla\Phi(d_s)\dd s.
\end{align*}
For fixed $w$ the second line vanishes, but the first retains the full, position-dependent Hessian $\nabla^2\Phi(d_s)$.
\end{proposition}
\begin{proof}
Apply the product and chain rules under the integral sign to \eqref{eq:geomforce}. Both integral terms are absolutely convergent under the stated differentiability domination.
\end{proof}

For example, the translation-invariant adaptation law
\begin{equation}\label{eq:adaptive}
 \W_\varepsilon(s;h)=w_0(s)
 \left[1+\varepsilon\frac{\norm{h(0)-h(-s)}^2}
 {1+\norm{h(0)-h(-s)}^2}\right]
\end{equation}
coincides with $w_0$ at constant histories. It satisfies $w_0\le\W_\varepsilon\le(1+\varepsilon)w_0$ and has the same linearization for every $\varepsilon\ge0$. If $m(s)=\rho e^{-s}$, $w_0(s)=e^{-s}$, and $\Phi(z)=e^{-\norm z^2/2}$, the scalar radial version of Theorem~\ref{thm:arrest} applies when $\rho+\kappa(1+\varepsilon)<1$ (choose $0<\gamma<1-\rho-\kappa(1+\varepsilon)$). The local linear boundary remains $\rho+\kappa=1$. The adaptive weight changes the nonlinear force throughout moving histories while keeping the equilibrium tangent operator fixed. The displayed global inequality gives an explicit sufficient arrest region.

For fixed $w$ write $H=-\nabla^2\Phi(0)\succeq0$ and $q(r)=\int_r^\infty w(s)\dd s$. The equilibrium variation is
\begin{align}\label{eq:restlinear}
 \dot y(t)={}&\int_0^\infty m(s)\dot y(t-s)\dd s\notag\\
 &+\kappa H\int_0^\infty w(s)
 [y(t)-y(t-s)]\dd s.
\end{align}
In a positive-eigenvalue direction $Hv=hv$, a nonzero modal exponent obeys
\begin{equation}\label{eq:char}
 1=\widehat m(\lambda)+\kappa h\widehat q(\lambda),
 \quad\widehat q(\lambda)=\int_0^\infty e^{-\lambda r}q(r)\dd r.
\end{equation}
The factor $\lambda=0$ removed in deriving \eqref{eq:char} corresponds to spatial translation. It must not be conflated with a growing or decaying mode.

\begin{proposition}[Classical positive-kernel threshold]\label{prop:spec}
Let $w,m\ge0$ be integrable and let $M_1=\int_0^\infty sw(s)\dd s<\infty$. Put $h_{\max}=\lambda_{\max}(H)>0$. If $M_0+\kappa h_{\max}M_1<1$, \eqref{eq:char} has no root with $\Re\lambda\ge0$ in any eigendirection. If $M_0+\kappa h_{\max}M_1>1$, it has a positive real root in the $h_{\max}$ eigendirection, unique on the positive real axis when the kernel is nontrivial. At equality the nontranslational characteristic factor reaches zero at $\lambda=0$.
\end{proposition}
\begin{proof}
The tail identity $\int_0^\infty q(r)\dd r=M_1$ follows from Tonelli's theorem. If the mass is below one, then for $\Re\lambda\ge0$,
\[
 |\widehat m(\lambda)+\kappa h\widehat q(\lambda)|
 \le M_0+\kappa hM_1<1,
\]
so \eqref{eq:char} is impossible. Above one, its right side is continuous, strictly decreasing (for a nontrivial positive kernel) along real $\lambda\ge0$, starts above one, and tends to zero by dominated convergence; it therefore crosses one exactly once. Equality gives the zero-frequency condition. This positive-convolution spectral calculation provides the equilibrium reference for the subsequent nonlinear moving-orbit analysis.
\end{proof}

\begin{remark}[Gaussian curvature away from equilibrium]
For $\Phi(z)=e^{-\norm z^2/(2\sigma^2)}$,
\[
 -\nabla^2\Phi(z)=\Phi(z)
 \left[\frac{I}{\sigma^2}-\frac{zz^\top}{\sigma^4}\right].
\]
Its transverse eigenvalues are $\Phi(z)/\sigma^2$, whereas its radial eigenvalue is $\Phi(z)(1-\norm z^2/\sigma^2)/\sigma^2$. The radial contribution changes sign at $\norm z=\sigma$. A characteristic equation involving only $H=-\nabla^2\Phi(0)$ cannot record this change.
\end{remark}

\section{Geometric traveling-speed theory}
We next leave the equilibrium tangent problem. Throughout this section $\W=w$ is fixed, $\Phi$ is radial, and $-\nabla\Phi(z)=a(\norm z)z$ with $a(r)>0$. Let
\begin{equation}\label{eq:Fa}
 F_a(r)=\int_0^\infty sw(s)a(rs)\dd s,
 \qquad r\ge0,
\end{equation}
and put $M_0=\int_0^\infty m(s)\dd s$. The scalar function $F_a$ is the \emph{geometric traveling response}: it samples the entire radial imprint across all ages, not merely the Hessian at the origin. Its physical meaning is simple. On a straight path moving at speed $r$, an imprint of age $s$ lies a distance $rs$ behind the current position. The response $F_a(r)$ integrates the corresponding forward force over all such past imprints. Persistent straight motion is possible precisely when this contribution balances the part of the velocity that the ordinary velocity-memory term does not reproduce.

\begin{theorem}[Exact traveling-speed condition]\label{thm:travel}
A complete nonzero uniform translation $x(t)=x_0+vt$, $v\ne0$, solves \eqref{eq:main} if and only if
\begin{equation}\label{eq:speed}
 \boxed{1-M_0=\kappa F_a(\norm v).}
\end{equation}
Consequently, if $M_0=1$, $a>0$ and $w$ is nontrivial, no nonzero uniform translation exists. If $0\le M_0<1$, every positive solution $r$ of \eqref{eq:speed} supplies the sphere of translating velocities $\{v:\norm v=r\}$.
\end{theorem}
\begin{proof}
For $x(t)=x_0+vt$, $\dot x=v$ and $\Delta_sx=sv$. The velocity convolution equals $M_0v$. Radiality makes the geometric force
\begin{align*}
 -\kappa\int_0^\infty w(s)\nabla\Phi(sv)\dd s
 &=\kappa v\int_0^\infty sw(s)a(s\norm v)\dd s\\
 &=\kappa vF_a(\norm v).
\end{align*}
Thus the vector equation is equivalent, after dividing by $v\ne0$, to \eqref{eq:speed}. All converse substitutions are exact. Positivity of $F_a$ gives the $M_0=1$ obstruction.
\end{proof}

\begin{remark}[What ``self-propulsive'' means]
On a straight path $\nabla S$ points backward and the \emph{force} $-\kappa\nabla S$ points forward. The impossibility of a translating solution when $M_0=1$ is a balance obstruction: the normalized velocity memory already reproduces $v$, so adding any strictly positive propulsive feedback makes the equation inconsistent. It does not mean that the memory force is a brake or that nonstationary coherence follows automatically.
\end{remark}

\subsection{Higher spatial geometry determines branch direction}
The exact balance equation determines possible speeds, but the small-speed geometry already reveals \emph{on which side} of a critical gain a moving branch may emerge. Two imprints can have the same curvature at the origin and hence the same equilibrium characteristic equation, while their next spatial derivatives bend their nonzero-speed branches in opposite directions. This is the first explicit indication that local spectral equivalence leaves genuine nonlinear geometry undetermined.

A first nonlinear refinement of equilibrium spectral analysis follows from the small-$r$ expansion of $F_a$. Denote $M_j=\int_0^\infty s^jw(s)\dd s$. If $a(0)=a_0>0$, then the equilibrium threshold is $\kappa_c=(1-M_0)/(a_0M_1)$.

\begin{theorem}[Local geometric branch orientation]\label{thm:orientation}
Suppose $0\le M_0<1$, $a\in C^4([0,\infty))$ is bounded with bounded derivatives through order four, $a'(0)=a'''(0)=0$, and $M_5<\infty$. Write
\begin{equation}\label{eq:aexp}
 a(r)=a_0+c_2r^2+O(r^4),\qquad a_0>0.
\end{equation}
Then
\begin{align}
 F_a(r)&=a_0M_1+c_2M_3r^2+O(r^4),\label{eq:Fexp}\\
 \kappa(r):=\frac{1-M_0}{F_a(r)}
 &=\kappa_c\left[1-\frac{c_2M_3}{a_0M_1}r^2+O(r^4)\right].\label{eq:ke}
\end{align}
If $c_2>0$ (respectively $c_2<0$), sufficiently small nonzero translation branches occur on the $\kappa<\kappa_c$ (respectively $\kappa>\kappa_c$) side. The sign of $c_2$ is invisible to the equilibrium characteristic equation, which uses only $a_0$.
\end{theorem}
\begin{proof}
Taylor's theorem and bounded fourth derivative give a uniform remainder of order $r^4s^4$ in $a(rs)-a_0-c_2r^2s^2$. Multiplying by $sw(s)$ and integrating is justified by $M_5<\infty$, yielding \eqref{eq:Fexp}. Because $F_a(0)=a_0M_1>0$, expansion of its reciprocal yields \eqref{eq:ke}. If $c_2$ has a strict sign, the quadratic term dominates the fourth-order remainder for all sufficiently small $r>0$, proving the asserted branch side.
\end{proof}

For a radial potential smooth near the origin, $c_2$ encodes its next nontrivial spatial derivative beyond the Hessian. This gives the local orientation of the exact moving branch. For $n\ge2$, rotational invariance generates the expected continuous family of velocity directions.

\subsection{A concrete pair with the same Hessian}
Let
\begin{equation}\label{eq:pair}
 \Phi_1(z)=e^{-\norm z^2/2},\qquad
 \Phi_2(z)=(3+\norm z^2)e^{-\norm z^2/2}.
\end{equation}
Then $a_1(r)=e^{-r^2/2}$ and $a_2(r)=(1+r^2)e^{-r^2/2}$. Both imprints are strictly decreasing for $r>0$, and $-\nabla^2\Phi_i(0)=I$. Yet
\[
 a_1(r)=1-\tfrac12r^2+O(r^4),\qquad
 a_2(r)=1+\tfrac12r^2+O(r^4).
\]
For $m(s)=\rho e^{-s}$, $w(s)=e^{-s}$ and $\rho<1$, Theorem~\ref{thm:orientation} therefore places their small-speed branches on opposite sides of $\kappa_c=1-\rho$.

\begin{proposition}[Explicit global branch multiplicity]\label{prop:pair}
For \eqref{eq:pair} and the exponential kernels above, $F_1$ decreases strictly from one to zero. Therefore $\Phi_1$ has exactly one positive translating speed when $\kappa>\kappa_c$ and none when $\kappa\le\kappa_c$. Conversely, $F_2(0)=1$, $F_2(r)>1$ for sufficiently small $r>0$, and $F_2(r)\to0$ as $r\to\infty$. Hence, writing $F_{2,\max}=\max_{r\ge0}F_2(r)>1$ and $\kappa_{\rm fold}=\kappa_c/F_{2,\max}$, every $\kappa\in(\kappa_{\rm fold},\kappa_c)$ admits at least two distinct positive translating speeds for $\Phi_2$.
\end{proposition}
\begin{proof}
The strictly decreasing radial factor $a_1(rs)$ makes $F_1$ strictly decreasing; dominated convergence gives its limit zero. For $a_2$, Theorem~\ref{thm:orientation} with $M_3=\int_0^\infty s^3e^{-s}\dd s=6$ gives
$F_2(r)=1+3r^2+O(r^4)>1$ near zero. Boundedness of $a_2$ and its decay to zero give $F_2(r)\to0$ by dominated convergence. The response is continuous and attains a maximum greater than one. For each horizontal level strictly between one and that maximum, the intermediate value theorem provides one crossing before a maximizing point and one after it. These crossings solve \eqref{eq:speed}; the two branches therefore follow directly from continuity of the exact traveling response.
\end{proof}

\section{Local-germ equivalence versus global geometry}
Agreement of Hessians still leaves infinitely many higher-order derivatives available to distinguish two imprints. We therefore ask a stronger question: could two fields exert \emph{exactly the same nonlinear force} for every sufficiently small displacement, and still permit different complete motions? The answer is affirmative within the spatial-memory class. We first establish the geometric construction for bounded-position histories, and later strengthen it to an open neighborhood in the \emph{same velocity-history phase space} used for the stability theory. The distinction between the two history topologies is important when temporal memory extends to arbitrarily large ages.

\begin{definition}[Force-germ equivalence]\label{def:germ}
Two imprints $\Phi_1,\Phi_2$ are \emph{force-germ equivalent at rest} if some $\delta>0$ satisfies
\[
 \nabla\Phi_1(z)=\nabla\Phi_2(z)
 \quad\text{for every }\norm z<\delta.
\]
For every fixed common weight $w$, their nonlinear forces then agree on any sup-norm neighborhood $\norm{h-h_\ast}_\infty<\delta/2$ of a constant history $h_\ast$. In particular, whenever their derivatives exist, all finite-order Fr\'echet derivatives at $h_\ast$ coincide.
\end{definition}

\begin{theorem}[Global orbit separation under identical local force germs]\label{thm:germ}
Let $n\ge1$, $m(s)=\rho e^{-s}$ with $0\le\rho<1$, and $w(s)=e^{-s}$. For every prescribed $\delta>0$ there exist two bounded, $C^\infty$, strictly radially decreasing imprints $\Phi_1,\Phi_2$ such that:
\begin{enumerate}
\item their gradients coincide for $\norm z<\delta$ and hence their entire nonlinear force operators agree on a uniform neighborhood of every constant history;
\item they have identical equilibrium linearizations and identical derivatives of every finite order at constant histories;
\item for an open interval of gains $0<\kappa<1-\rho$, the first system admits no nonzero uniform translation, while the second admits at least two distinct nonzero translating speeds.
\end{enumerate}
Both systems use the \emph{same} time kernels $m,w$. Therefore no equilibrium transfer function, finite-order expansion, or even complete local force germ at rest distinguishes their global translating orbit families.
\end{theorem}
\begin{proof}
Choose any nonnegative nonzero $b\in C_c^\infty((\delta,2\delta))$. Define
\[
 a_1(r)=e^{-r^2/2},\qquad a_2(r)=a_1(r)+Cb(r),
 \qquad C>0.
\]
Both functions are smooth, positive, bounded, and tend to zero as $r\to\infty$. Define their radial potentials by
\begin{equation}\label{eq:construct}
 \Phi_i(z)=\int_{\norm z}^\infty r a_i(r)\dd r,
 \quad i=1,2.
\end{equation}
The integrals converge because $b$ has compact support and $a_1$ decays rapidly. Since $a_1(r)=e^{-r^2/2}$, $\Phi_1(z)=e^{-\norm z^2/2}$. The perturbation of $\Phi_1$ is a smooth radial function that is constant when $\norm z\le\delta$ and zero when $\norm z\ge2\delta$; hence $\Phi_2$ is smooth at the origin as well. Direct differentiation gives
\[
 -\nabla\Phi_i(z)=a_i(\norm z)z,
\]
so both imprints decrease strictly for every positive radius. Their gradients agree whenever $\norm z<\delta$.

If $\norm{h-h_\ast}_\infty<\delta/2$, then $\norm{h(0)-h(-s)}<\delta$ for all $s$, proving that their integrands and therefore their forces are \emph{exactly equal} on this history neighborhood. All local derivatives agree as a direct consequence; no Taylor approximation is involved.

The first response $F_1(r)=\int_0^\infty se^{-s}a_1(rs)\dd s$ decreases strictly from $1$ to $0$, so it has no positive translating root when $\kappa<1-\rho$. For the second,
\begin{equation}\label{eq:responsebump}
 F_2(r)=F_1(r)+C B(r),\quad
 B(r)=\int_0^\infty se^{-s}b(rs)\dd s.
\end{equation}
Fix any $r_0>0$. Since $b\ge0$ is nonzero and $e^{-s}>0$, we have $B(r_0)>0$. Choose $C$ large enough that $F_2(r_0)>1$. Also $F_2(0)=1$ because $b(0)=0$, and $F_2(r)\to0$ as $r\to\infty$ by dominated convergence (both $a_i$ are bounded and decay). Thus $F_{2,\max}>1$. For every
\[
 \kappa\in\left(\frac{1-\rho}{F_{2,\max}},\,1-\rho\right)
\]
the speed level $(1-\rho)/\kappa$ lies strictly between one and the maximum. Continuity yields at least one positive crossing on either side of a maximizer. Theorem~\ref{thm:travel} turns both into exact nonzero uniform translations. The first system has none there. This proves all three assertions.
\end{proof}

\begin{remark}[Why this is stronger than identical spectra]
The construction gives exact local agreement of the force in bounded-position histories. The subsequent velocity-history formulation makes the stationary and translating solutions elements of one Banach space and provides the strengthened stability classification. The potentials may differ by an additive constant near the origin; their gradients agree there exactly.
\end{remark}

\section{Variational instability of translating branches}
Finding a nonzero root of the traveling-speed equation proves that a straight motion \emph{exists}; it does not establish whether a nearby motion stays close. The next theorem addresses this second question directly along the complete moving trajectory. Differentiation of the force at the nonzero displacements $rs$ produces the longitudinal factor $b_r(s)$, and the slope of the geometric response furnishes a sign test for a growing variational mode. This connects the geometry of the imprint to the \emph{moving-orbit} spectrum, which need not resemble the equilibrium spectrum.

\begin{theorem}[Positive longitudinal exponent on rising branches]\label{thm:travelunstable}
Assume the fixed-kernel radial model of Theorem~\ref{thm:travel}, with $0\le M_0<1$, $m,w\ge0$ integrable, $sw(s)\in L^1$, and $a\in C^1([0,\infty))$ such that $a$ and $ra'(r)$ are bounded. Let $r>0$ solve \eqref{eq:speed} and suppose $F_a$ is differentiable there with
\begin{equation}\label{eq:slopepos}
 F_a'(r)>0.
\end{equation}
Then the linear variational equation along any corresponding uniform translation has at least one \emph{positive real longitudinal characteristic exponent}. In particular, the small-speed subcritical branch associated with $\Phi_2$ in \eqref{eq:pair} possesses a growing longitudinal variational mode.
\end{theorem}
\begin{proof}
	Write $v=re$, $\norm e=1$, and perturb the complete
	translating orbit $x(t)=x_0+vt$ by $y(t)=\xi(t)e$.
	Differentiating the radial force map
	$z\mapsto a(\norm z)z$ at $z=sv$ gives the longitudinal
	scalar multiplier
	\[
	b_r(s)=a(rs)+rs a'(rs).
	\]
	Consequently, the longitudinal variational equation is
	\begin{align}\label{eq:longvar}
		\dot\xi(t)={}&
		\int_0^\infty m(s)\dot\xi(t-s)\dd s\notag\\
		&+\kappa\int_0^\infty w(s)b_r(s)
		[\xi(t)-\xi(t-s)]\dd s.
	\end{align}
	This equation has constant coefficients because the
	reference increments $\Delta_sx(t)=sv$ are independent
	of time.
	
	To examine longitudinal instability, consider an
	exponential perturbation $\xi(t)=e^{\lambda t}$
	with $\lambda>0$. Substitution into
	\eqref{eq:longvar} and division by $\lambda$
	lead to the characteristic function
	\begin{align*}
		D_r(\lambda):={}&1-\widehat m(\lambda)\\[-1pt]
		&-\kappa\int_0^\infty w(s)b_r(s)
		\frac{1-e^{-\lambda s}}{\lambda}\dd s,
	\end{align*}
	where
	\[
	\widehat m(\lambda)
	=\int_0^\infty m(s)e^{-\lambda s}\dd s.
	\]
	The corresponding characteristic equation is
	\begin{equation}\label{eq:Dlambda}
		D_r(\lambda)=0.
	\end{equation}
	
	We now prove that \eqref{eq:Dlambda} possesses
	at least one positive real root.
	
	First, the boundedness assumptions and
	$sw\in L^1(0,\infty)$ ensure that $D_r$ extends
	continuously to $[0,\infty)$, with
	\[
	\left.
	\frac{1-e^{-\lambda s}}{\lambda}
	\right|_{\lambda=0}=s.
	\]
	Differentiation under the integral defining
	$F_a$ is justified by the same assumptions
	and yields
	\[
	\int_0^\infty sw(s)b_r(s)\dd s
	=F_a(r)+rF_a'(r).
	\]
	Using the exact traveling-branch balance
	$1-M_0=\kappa F_a(r)$, we therefore obtain
	\begin{equation}\label{eq:D0}
		\begin{aligned}
			D_r(0)
			&=1-M_0-\kappa[F_a(r)+rF_a'(r)]\\
			&=-\kappa rF_a'(r)<0,
		\end{aligned}
	\end{equation}
	where the strict inequality follows from
	\eqref{eq:slopepos}.
	
	Next, since $m\in L^1(0,\infty)$,
	\[
	\lim_{\lambda\to+\infty}
	\widehat m(\lambda)=0.
	\]
	Moreover, for every $\lambda>0$ and $s\geq0$,
	\[
	0\leq
	\frac{1-e^{-\lambda s}}{\lambda}
	\leq s.
	\]
	The quotient tends to zero as
	$\lambda\to+\infty$ for each fixed $s$.
	Since $sw(s)|b_r(s)|$ is integrable,
	the dominated convergence theorem gives
	\[
	\lim_{\lambda\to+\infty}
	\int_0^\infty w(s)b_r(s)
	\frac{1-e^{-\lambda s}}{\lambda}\dd s=0.
	\]
	Consequently,
	\[
	\lim_{\lambda\to+\infty}D_r(\lambda)=1>0.
	\]
	
	Thus $D_r$ is continuous on $[0,\infty)$,
	satisfies $D_r(0)<0$, and tends to a strictly
	positive limit as $\lambda\to+\infty$.
	By the intermediate value theorem,
	there exists $\lambda_+>0$ such that
	\[
	D_r(\lambda_+)=0.
	\]
	The corresponding perturbation
	$\xi(t)=e^{\lambda_+t}$ is a growing
	longitudinal mode of the translating solution.
	
	Finally, for $\Phi_2$ in \eqref{eq:pair},
	the expansion
	\[
	F_2(r)=1+3r^2+O(r^4)
	\]
	implies $F_2'(r)>0$ for all sufficiently
	small $r>0$. The associated moving branch
	therefore satisfies \eqref{eq:slopepos}
	and possesses a positive real longitudinal
	characteristic exponent. Moreover, its
	gain satisfies
	\[
	\kappa(r)=\frac{1-\rho}{F_2(r)}<1-\rho
	\]
	for sufficiently small $r>0$, confirming
	that this unstable branch occurs in the
	subcritical gain regime.
\end{proof}

\begin{corollary}[Fold degeneracy]\label{cor:fold}
At a differentiable turning point $r_\ast>0$ of the speed response with $F_a'(r_\ast)=0$, the nontranslational longitudinal characteristic factor in \eqref{eq:Dlambda} satisfies $D_{r_\ast}(0)=0$. The additional zero-frequency degeneracy is a necessary spectral signature of such a fold, but does not by itself determine the stability of the descending branch.
\end{corollary}
\begin{proof}
Equation \eqref{eq:D0} holds without the strict inequality. Substituting $F_a'(r_\ast)=0$ yields $D_{r_\ast}(0)=0$. The limitation in the conclusion follows because no sign determination for the complete complex spectrum follows from this single scalar equality.
\end{proof}

\begin{remark}[Transverse modes and symmetry]
For $n\ge2$, transverse perturbations have multiplier $a(rs)$ in place of $b_r(s)$. Their divided characteristic factor has zero-frequency value $1-M_0-\kappa F_a(r)=0$ on a translating branch, consistent with rotational degeneracy of the velocity sphere. For $\Re\lambda>0$, positivity and the branch balance yield
\[
 \left|\widehat m(\lambda)+\kappa\int_0^\infty w(s)a(rs)
 \frac{1-e^{-\lambda s}}{\lambda}\dd s\right|<1,
\]
under the nondegenerate positive-kernel assumptions. Thus no transverse characteristic exponent with positive real part arises from this scalar factor. This does not imply orbital asymptotic stability: the neutral translation and rotation modes must be treated on an appropriate quotient phase space.
\end{remark}

\begin{remark}[Limits of the instability assertion]
A positive real exponent characterizes the moving-orbit linearization. Nonlinear orbital instability can be investigated through differentiable evolution maps on fading-history spaces; the present moving-branch classification also provides a direct nonlinear attraction theorem for the descending one-dimensional branch.
\end{remark}

\section{Moving-branch stability in the velocity phase space}\label{sec:branchstab}
The existence of a positive variational exponent explains instability on certain increasing-response branches. To describe the other side of the picture, we need a genuinely nonlinear estimate: does a small perturbation of a moving speed decay under the \emph{complete} hereditary equation? The next criterion answers this question in one dimension by controlling the derivative of the spatial force along the entire translating history. The contraction is imposed on the unweighted initial-history ball; the resulting trajectories converge in fading-memory norms.

We now connect the traveling-speed response to nonlinear stability of the corresponding velocity history. For a scalar radial interaction put
\[
 g(d)=a(|d|)d,\qquad
 b_r(s)=g'(rs)=a(rs)+rs a'(rs).
\]
The directional derivative $b_r$ samples the full radial force law along the complete moving history. Define
\begin{equation}\label{eq:Br}
 B_r=\int_0^\infty sw(s)|b_r(s)|\dd s.
\end{equation}

\begin{theorem}[A two-sided moving-branch criterion]\label{thm:twosided}
Let $m,w\ge0$ have finite first memory moment and let $a\in C^1([0,\infty))$ be positive with $a$ and $ua'(u)$ bounded. Consider an exact translating speed $r>0$ satisfying $1-M_0=\kappa F_a(r)$.
\begin{enumerate}
\item If $F_a'(r)>0$, the longitudinal variation admits a positive real characteristic exponent.
\item If $M_0+\kappa B_r<1$, the longitudinal characteristic equation has no roots in $\Re\lambda\ge0$. If in addition $m$ and $sw$ have positive exponential moments and $g'$ is bounded and continuous, the scalar velocity equation has a locally exponentially attracting constant-speed solution $v(t)\equiv r$: for every compatible prehistory $v_-\in\mathcal V$ satisfying $\norm{v_--r}_\infty<\varepsilon$ for a sufficiently small $\varepsilon>0$, there exist $C,\gamma>0$ such that
\begin{equation}\label{eq:nonlinearstable}
 |v(t)-r|\le C e^{-\gamma t},\qquad
 x(t)-rt\longrightarrow x_\infty.
\end{equation}
The associated full-history perturbation converges to zero in every fading supremum norm $\norm z_\chi=\sup_{s\ge0}e^{-\chi s}|z(-s)|$, $\chi>0$. Thus the initial perturbation is controlled in the \emph{unweighted} supremum norm, while convergence of the resulting histories is established in \emph{fading-memory} norms.
\end{enumerate}
\end{theorem}
\begin{proof}
The first assertion is Theorem~\ref{thm:travelunstable}. In the longitudinal variation the exponential velocity mode satisfies $D_r(\lambda)=0$ from \eqref{eq:Dlambda}. For $\Re\lambda\ge0$ the elementary integral identity
\[
 \frac{1-e^{-\lambda s}}{\lambda}=\int_0^s e^{-\lambda u}\dd u
\]
gives $|(1-e^{-\lambda s})/\lambda|\le s$ for $\lambda\ne0$, with its continuous value $s$ at zero. Consequently, the modulus of the complete complex quantity satisfies
\begin{equation}\label{eq:spectralmodulus}
\begin{aligned}
 &\left|\widehat m(\lambda)+\kappa\int_0^\infty
 w(s)b_r(s)\frac{1-e^{-\lambda s}}{\lambda}\dd s\right|\\
 &\hspace{1.5cm}\le M_0+\kappa B_r<1.
\end{aligned}
\end{equation}
Thus the nontranslational characteristic factor has no root in the closed right half-plane.

For the nonlinear conclusion work in the common phase space $\mathcal V$ of Section~\ref{sec:phase}. Let $\bar v(\theta)\equiv r$ and choose $\varepsilon>0$. For a history in the $\varepsilon$-ball about $\bar v$, the mean value theorem gives the local Lipschitz bound
\begin{align*}
 &|\mathfrak T_\Phi(v)-\mathfrak T_\Phi(u)|
 \le q_\varepsilon(0)\norm{v-u}_\infty,\\
 &q_\varepsilon(0)=M_0+\kappa\int_0^\infty sw(s)
 \sup_{|d-rs|\le s\varepsilon}|g'(d)|\dd s.
\end{align*}
Boundedness and continuity of $g'$ and dominated convergence imply $q_\varepsilon(0)\to M_0+\kappa B_r<1$ as $\varepsilon\downarrow0$. Fix $\varepsilon$ so that $q_\varepsilon(0)<1$. For a prescribed compatible past with $\norm{v_--r}_\infty<\varepsilon$, the fixed-point map on the closed $\varepsilon$-ball of bounded future velocities satisfying $v(0)=v_-(0)$ is invariant because $\mathfrak T_\Phi(\bar v)=r$, and is a contraction with constant $q_\varepsilon(0)$. Thus the solution exists uniquely within that ball and stays there.

By the positive exponential-moment assumption, continuity at $\gamma=0$ allows a sufficiently small $\gamma>0$ such that
\begin{align*}
 q_\varepsilon(\gamma)={}&\int_0^\infty m(s)e^{\gamma s}\dd s\\
 &+\kappa\int_0^\infty w(s)
 \sup_{|d-rs|\le s\varepsilon}|g'(d)|
 \frac{e^{\gamma s}-1}{\gamma}\dd s<1.
\end{align*}
The difference $z(t)=v(t)-r$ then satisfies the same weighted comparison inequality as in Theorem~\ref{thm:arrest}, with coefficient $q_\varepsilon(\gamma)$. Writing $Z_-=\sup_{t\le0}|z(t)|$ and taking successive suprema on $[0,T]$ yields $|z(t)|\le Z_-e^{-\gamma t}$. Hence $\int_0^\infty |z(t)|\dd t<\infty$ and $x(t)-rt$ has a finite limit. Finally, splitting a fading-memory history into $s\le t$ and $s>t$ gives
\[
 \norm{z_t}_\chi\le Z_-e^{-\min\{\chi,\gamma\}t},
\]
with the harmless equality case understood directly. This establishes the claimed attraction on the reduced velocity phase space.
\end{proof}

\begin{remark}[Phase-space topology for moving-branch attraction]\label{rem:topology}
For infinite prehistory, convergence of the current velocity does not imply exponential decay in the unweighted supremum norm on $\mathcal V$. A perturbation supported sufficiently far in the past remains present in that norm for every finite time. Theorem~\ref{thm:twosided} requires small compatible initial perturbations in the unweighted norm $\norm{v_--r}_\infty$. It then proves exponential decay of the current speed and convergence of the complete history in weighted fading-memory norms. Smallness solely in a fading-memory norm is not among the theorem's hypotheses: such a weighted ball can contain large remote-past deviations.
\end{remark}

\begin{proposition}[Transverse spectrum and velocity-direction symmetry]\label{prop:transverse}
Suppose $n\ge2$ and $m,w\ge0$, with $F_a(r)>0$ on a translating branch. For each transverse perturbation, the divided characteristic factor is
\begin{equation}\label{eq:transverseD}
 D_\perp(\lambda)=1-\widehat m(\lambda)-\kappa
 \int_0^\infty w(s)a(rs)
 \frac{1-e^{-\lambda s}}{\lambda}\dd s.
\end{equation}
Then $D_\perp(0)=0$, while $D_\perp(\lambda)\ne0$ for $\Re\lambda\ge0$, $\lambda\ne0$. Thus the transverse velocity spectrum contains the neutral rotation mode and no exponent in the open right half-plane.
\end{proposition}
\begin{proof}
At $\lambda=0$, the branch balance yields $D_\perp(0)=1-M_0-\kappa F_a(r)=0$. For $\Re\lambda>0$, positivity and $|\int_0^s e^{-\lambda u}\dd u|<s$ give a strict modulus bound below $M_0+\kappa F_a(r)=1$. The same strict inequality holds for nonzero purely imaginary $\lambda$: the modulus of the integral of the nonconstant unit-modulus function $u\mapsto e^{-\lambda u}$ is strictly smaller than $s$ for every $s>0$. Integration against a nontrivial nonnegative $w$ completes the proof.
\end{proof}

\section{Geometric coexistence with identical local forces}\label{sec:strongsep}
We now reach the main structural result. Imagine observing a system only while its velocity histories remain close to rest. Even with exact access to the nonlinear memory force throughout that local neighborhood, one could not distinguish the two systems constructed below. Yet their global outcomes differ: one \emph{always stops} from every admissible bounded velocity prehistory; the other can either settle to rest or approach a persistent straight motion, and it also possesses a translating branch with a growing longitudinal perturbation.

The theorem makes this comparison under restrictive shared conditions. Both imprints are smooth, bounded, strictly radially decreasing, and coupled to the same positive temporal weights and gain. Only their finite-displacement geometry differs. For the first imprint, a global Lipschitz estimate controls all histories. For the second, a shaped response at nonzero speeds creates translating branches and supplies the derivative bounds needed for attraction. Thus one local force germ is compatible with demonstrably distinct global dynamics in the exact hereditary equation.

\begin{theorem}[Identical local forces: global arrest versus rest--motion coexistence]\label{thm:geostable}
Fix $0<\rho<1$ and $m(s)=\rho e^{-s}$. There exist nonnegative $C_c^\infty$ temporal weights $w$ normalized by $\int_0^\infty sw(s)\dd s=1$, two bounded smooth strictly radially decreasing imprints $\Phi_1,\Phi_2$ on $\R$, and a common gain $0<\kappa<1-\rho$ with the following properties.
\begin{enumerate}
\item The memory forces $\mathfrak F_{\Phi_1}$ and $\mathfrak F_{\Phi_2}$ coincide exactly on an open neighborhood of the zero history in the common velocity-history space $\mathcal V$. All their local Fr\'echet derivatives at rest agree.
\item Every admissible solution of the $\Phi_1$ system with bounded velocity prehistory exists globally, has exponentially decaying velocity, and converges to a history-dependent constant position. In particular, the system has no nonzero uniformly translating solution.
\item For the $\Phi_2$ system, the resting velocity $v\equiv0$ is locally exponentially attracting in fading-memory topology. There are at least two distinct positive uniformly translating speeds $r_-<r_+$: the lower-speed branch has a positive real longitudinal characteristic exponent, while the upper-speed branch is locally exponentially attracting in the scalar velocity dynamics, and $x(t)-r_+t$ converges to a constant along nearby compatible solutions.
\end{enumerate}
The same local nonlinear force germ therefore supports globally arrested dynamics in one system and coexistence of locally attracting rest and motion with an additional spectrally unstable moving branch in the other. The assertion concerns the specified branches and does not impose an exhaustive global attractor classification.
\end{theorem}
\begin{proof}
Define the standard smooth step function
\begin{equation}\label{eq:smoothstep}
 \vartheta(t)=\begin{cases}
 0,&t\le0,\\
 \displaystyle\frac{e^{-1/t}}{e^{-1/t}+e^{-1/(1-t)}},&0<t<1,\\
 1,&t\ge1,
 \end{cases}
\end{equation}
and let
\begin{equation}\label{eq:chi}
 \chi(u)=\vartheta\!\left(\frac{u-1/4}{3/5-1/4}\right)
 \vartheta\!\left(\frac{2-u}{2-7/5}\right).
\end{equation}
Then $\chi\in C_c^\infty(\R), \ \operatorname{supp}\chi\subset[1/4,2]$, $0\le\chi\le1$, and $\chi=1$ on $[3/5,7/5]$. Define
\begin{align}
 a_1(u)&=e^{-u^2/2},\\
 a_2(u)&=e^{-u^2/2}\bigl(1+u^2\chi(u)\bigr),\label{eq:strongpair}\\
 \Phi_i(x)&=\int_{|x|}^\infty u a_i(u)\dd u.
\end{align}
These are smooth, bounded, strictly decreasing radial imprints with $-\Phi'_i(x)=a_i(|x|)x$ and equal gradients for $|x|<1/4$. For $\eta>0$, introduce
\begin{equation}\label{eq:temporalshape}
\begin{aligned}
\beta(t)&=\begin{cases}
e^{-1/(1-t^2)},&|t|<1,\\
0,&|t|\ge1,
\end{cases}\\
w_\eta(s)&=\frac{\beta((s-1)/\eta)}
{\displaystyle\int_0^\infty u\beta((u-1)/\eta)\dd u}.
\end{aligned}
\end{equation}
Thus $w_\eta\ge0$ is smooth, is supported in $[1-\eta,1+\eta]$, and satisfies $\int_0^\infty sw_\eta(s)\dd s=1$. We choose $\eta>0$ sufficiently small in the following argument.

Since $w_\eta$ has bounded support, any history with $\norm v_\infty<1/[4(1+\eta)]$ satisfies $|\mathsf d_sv|<1/4$ for every $s$ in the support of $w_\eta$. Hence the memory forces are identical throughout this \emph{open velocity-history neighborhood}, proving the first assertion directly in $\mathcal V$.

Let $r_+=6/5$ and set $F_i(r)=\int sw_\eta(s)a_i(rs)\dd s$. Because $a_2(u)=(1+u^2)e^{-u^2/2}$ throughout a neighborhood of $u=6/5$, the elementary inequalities
\begin{align*}
 a_2(6/5)&>1,\qquad a'_2(6/5)<0,\\
 0&<a_2(6/5)+(6/5)a'_2(6/5)\\
  &<a_2(6/5)
\end{align*}
are strict. By continuity, for sufficiently small $\eta$ they hold at $u=(6/5)s$ for every $s$ in the support of $w_\eta$. In particular,
\[
 F_2(r_+)>1,\qquad
 \int sw_\eta(s)|b_{r_+}(s)|\dd s<F_2(r_+).
\]
Fix $\kappa=(1-\rho)/F_2(r_+)<1-\rho$.

For the first imprint, put $g_1(d)=de^{-d^2/2}$. Then
\[
 g'_1(d)=(1-d^2)e^{-d^2/2},\qquad
 \sup_{d\in\R}|g'_1(d)|=1.
\]
Since $M_0=\rho$, $M_1=1$, and $\kappa<1-\rho$, the existence condition $M_0+\kappa M_1<1$ holds. Proposition~\ref{prop:exist} gives global solutions for bounded continuous velocity prehistories, with $C^1$ regularity at the join for compatible initial data. Because $m(s)=\rho e^{-s}$ and $w_\eta$ is compactly supported, the left-hand side of \eqref{eq:Qgamma} converges to $\rho+\kappa<1$ as $\gamma\downarrow0$. For sufficiently small $\gamma>0$, Theorem~\ref{thm:arrest} therefore gives exponential velocity arrest and convergence to a constant position for every admissible solution of the first system.

The two imprints agree near zero, with $a_1(0)=a_2(0)=1$. The local nonlinear Lipschitz index of the $\Phi_2$ velocity operator at $v\equiv0$ is thus $\rho+\kappa<1$. The local contraction and exponentially weighted comparison of Theorem~\ref{thm:twosided}, centered at $r=0$, yield exponential decay of small compatible speed perturbations and convergence of their histories in fading-memory norms. This establishes the required attraction of rest in the second system.

The speed $r_+$ is an exact translating solution by Theorem~\ref{thm:travel}. Its contraction index satisfies
\[
 \rho+\kappa B_{r_+}
 <\rho+\kappa F_2(r_+)=1.
\]
Both $m$ and $w_\eta$ have exponential moments, and the nonlinear force derivative is bounded and continuous. Theorem~\ref{thm:twosided} proves local exponential attraction of this speed.

At least one additional positive-speed solution follows from the strict inequalities
\[
 a_2(7/10)<a_2(6/5)<a_2(9/10),
\]
which hold in the region where $\chi=1$. By concentration of the normalized weights near $s=1$, the same inequalities hold for $F_2$ when $\eta$ is sufficiently small. Continuity gives a solution $r_-\in(7/10,9/10)$ of $F_2(r_-)=F_2(r_+)$. Throughout that interval and on the support of $w_\eta$ one has $rs\in(3/5,1)$, so
\[
 a'_2(rs)=rs\bigl(1-r^2s^2\bigr)e^{-r^2s^2/2}>0.
\]
Therefore $F'_2(r_-)>0$. Theorem~\ref{thm:travelunstable} gives a positive real longitudinal exponent.

Finally, $a_1$ is strictly decreasing and $F_1(0)=1$; hence $F_1(r)<1$ for every $r>0$. Our chosen gain makes the required speed level $(1-\rho)/\kappa=F_2(r_+)>1$, so the $\Phi_1$ system has no positive translating speed. Evenness gives the corresponding negative-speed conclusions. All strict inequalities are preserved for sufficiently small support width $\eta$, completing the proof.
\end{proof}

\begin{remark}[Normalized and subnormalized velocity memory]\label{rem:normalization}
The coexistence mechanism in Theorem~\ref{thm:geostable} uses the subnormalized velocity-memory mass $M_0=\rho<1$. The traveling balance \eqref{eq:speed} is
\[
 1-M_0=\kappa F_a(r),\qquad r>0.
\]
For a strictly repulsive radial imprint ($a>0$) and a nontrivial positive temporal weight, $F_a(r)>0$. Therefore the normalized case $M_0=1$, including $m(s)=\mu e^{-\mu s}$, admits no nonzero uniform translation in this class, as proved in Theorem~\ref{thm:travel}. When $M_0<1$, the positive deficit $1-M_0$ can balance the self-generated propulsive memory term, allowing the nonzero branches constructed above. The two regimes have distinct exact traveling-balance structures.
\end{remark}

Theorem~\ref{thm:geostable} establishes a precise \emph{local-to-global non-identifiability}: equality of the full nonlinear force in a neighborhood of rest does not identify the family of complete moving solutions, even when the temporal kernels, gain, and radial monotonicity constraints are fixed. Its global-arrest conclusion concerns every bounded admissible prehistory of the first system; its attraction conclusions for the second are local and apply to the identified rest and upper translating branches. A positive longitudinal exponent is proved for the lower branch. The theorem does not characterize basin boundaries or exclude further invariant motions.

The construction is visualized in Fig.~\ref{fig:branches} using the explicit kernels \eqref{eq:chi} and \eqref{eq:temporalshape}. Both panels plot exact \emph{constant-speed balance information}; neither panel is an evolution equation for the instantaneous speed. The attraction and spectral conclusions follow independently from Theorems~\ref{thm:arrest}, \ref{thm:travelunstable}, and \ref{thm:twosided}. In particular, an autonomous one-dimensional speed portrait would not represent the hereditary evolution \eqref{eq:abstract_velocity}.

\begin{figure*}[t]
\centering
\includegraphics[width=0.98\textwidth]{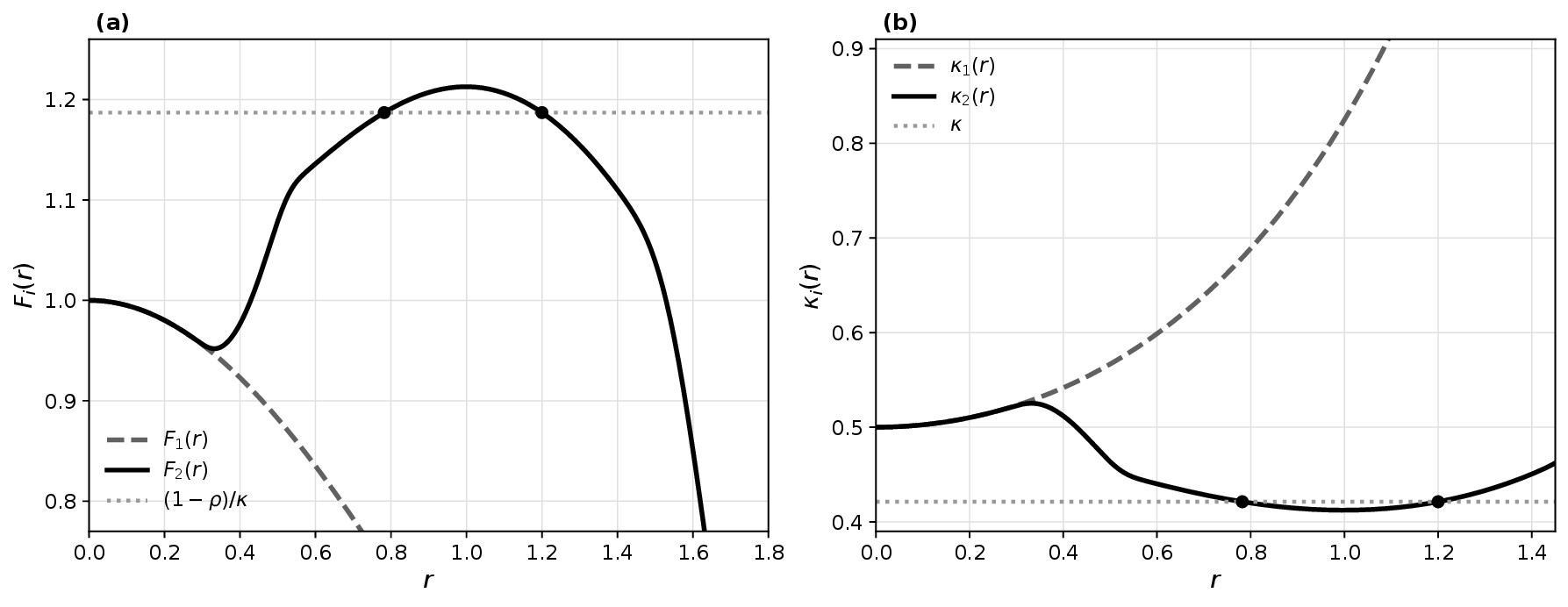}
\caption{Exact geometric traveling-speed branches for the construction in Theorem~\ref{thm:geostable}; here $\rho=1/2$, $\eta=0.06$, $r_+=6/5$, and $\kappa=(1-\rho)/F_2(r_+)\approx0.421243$. (a) The responses $F_i(r)=\int sw_\eta(s)a_i(rs)\,ds$ and their common balance level. (b) Gain--speed curves $\kappa_i(r)=(1-\rho)/F_i(r)$; intersections with the horizontal gain yield translating solutions. For $\Phi_2$, the identified positive speeds are $r_-\approx0.782069$ and $r_+=1.2$. The positive longitudinal characteristic exponent of the
	lower branch and the local exponential attraction of the
	upper branch are proved in Theorems~\ref{thm:travelunstable}
	and~\ref{thm:twosided}, respectively, and cannot be
	inferred from the traveling-speed curves alone.}
\label{fig:branches}
\end{figure*}

An infinitely long memory tail produces a subtler situation. A tiny constant speed still generates large displacements if we look far enough into the past. Consequently, two imprints that agree at small spatial distances need not generate identical force operators on any open \emph{velocity-history} ball. The distant past is strongly suppressed by an exponential temporal weight, however. The next theorem quantifies the result: the two forces may differ for arbitrarily small velocity histories, yet their difference decays faster than every algebraic power of the history amplitude. This provides an exact meaning for the inability of finite-order local tests to detect certain finite-displacement changes.

\begin{theorem}[Exponential flatness of distinct memory forces]\label{thm:flat}
Let $w(s)=e^{-s}$, $a_1(u)=e^{-u^2/2}$, and $a_2(u)=a_1(u)+Cb(u)$ with $C>0$ and $b\in C_c^\infty((\delta,2\delta))$ nonnegative and nonzero. Let $\mathfrak F_i$ be their forces on $\mathcal V$, and put $\varepsilon=\norm v_\infty$. Then for $\varepsilon>0$,
\begin{equation}\label{eq:flat}
 \norm{\mathfrak F_2(v)-\mathfrak F_1(v)}
 \le C\norm b_\infty(\delta+\varepsilon)
 e^{-\delta/\varepsilon}.
\end{equation}
The two force maps are $C^\infty$ near zero with identical Fr\'echet derivatives of every order there. Nevertheless, they differ at some history in every open velocity ball about zero.
\end{theorem}
\begin{proof}
Set $d_s=\mathsf d_sv$. Then $\norm{d_s}\le\varepsilon s$. Since $b$ is supported at radii at least $\delta$, a nonzero difference requires $s\ge\delta/\varepsilon$. Thus
\begin{align*}
 \norm{\mathfrak F_2(v)-\mathfrak F_1(v)}
 &\le C\norm b_\infty\varepsilon
 \int_{\delta/\varepsilon}^\infty se^{-s}\dd s\\
 &=C\norm b_\infty(\delta+\varepsilon)e^{-\delta/\varepsilon}.
\end{align*}
For each $k\ge1$, the $k$th history derivative of the smooth compactly supported function $d\mapsto b(\norm d)d$ is bounded; composition with $v\mapsto\mathsf d_sv$ contributes at most $s^k$. As $\int_0^\infty s^ke^{-s}\dd s<\infty$, repeated differentiation under the integral gives a $C^\infty$ force map. The difference estimate is $o(\varepsilon^N)$ for every integer $N$, forcing all its derivatives at zero to vanish.

On the other hand, choose any constant velocity history $v(\theta)\equiv re$ with $r>0$. Then the force difference is
\[
 Cr e\int_0^\infty se^{-s}b(rs)\dd s,
\]
which is nonzero for every $r>0$, since $b$ is nontrivial and the integrand is positive on a nonempty interval. Such histories occur in every ball about zero. Hence no open velocity ball supports exact equality of these two infinitely supported-memory force operators.
\end{proof}

\begin{example}[Reproducible smooth bump and exponentially flat force]\label{ex:bump}
For the illustration in Fig.~\ref{fig:flatness}, take $\delta=1$, $C=6$, and
\begin{equation}\label{eq:bump}
 b(u)=\begin{cases}
 \displaystyle\exp\!\left(-\frac{1}{(u-1)(2-u)}\right),&1<u<2,\\
 0,&\text{otherwise}.
 \end{cases}
\end{equation}
Then $\|b\|_\infty=e^{-4}$. For constant speed histories $v(\theta)\equiv\varepsilon>0$ the exact force difference is
\begin{equation}\label{eq:flatquad}
 |\mathfrak F_2(v)-\mathfrak F_1(v)|
 =\frac{C}{\varepsilon}\int_1^2 u e^{-u/\varepsilon}b(u)\dd u.
\end{equation}
The upper bound of Theorem~\ref{thm:flat} is $Ce^{-4}(1+\varepsilon)e^{-1/\varepsilon}$. Theorem~\ref{thm:germ} guarantees differing translating branches for sufficiently large values of $C$ and corresponding subcritical gains; the flatness estimate itself holds for every $C>0$.
\end{example}

\begin{figure*}[t]
\centering
\includegraphics[width=0.98\textwidth]{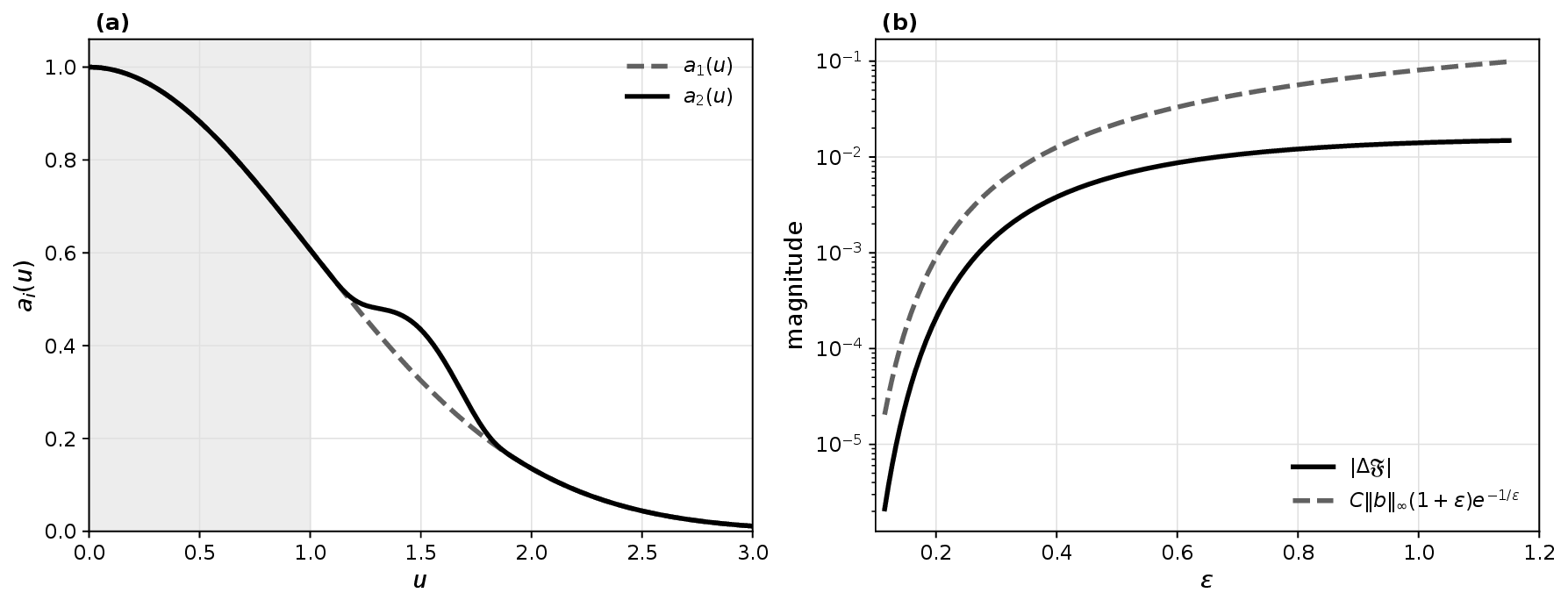}
\caption{Smooth geometric imprint perturbation and flatness for the explicit bump \eqref{eq:bump}. (a) The two radial force factors agree for $0\le u\le1$, while their difference is supported on $1<u<2$. (b) The exact constant-speed force difference \eqref{eq:flatquad} for $w(s)=e^{-s}$ and its exponential upper bound. Every curve is computed directly by one-dimensional quadrature; no time-domain simulation is used.}
\label{fig:flatness}
\end{figure*}

\begin{remark}[Smooth versus analytic locality]\label{rem:analytic}
Theorem~\ref{thm:flat} uses a $C^\infty$ bump perturbation. In the real-analytic class, equality of the radial force factor on a nontrivial interval forces equality on the connected domain by analytic continuation. The flat non-identifiability mechanism is therefore specific to the smooth non-analytic setting.
\end{remark}

\section{How classical control and delay theory emerge}
The same equations admit several familiar descriptions after an appropriate specialization. Their linearization is a Volterra convolution; a low-frequency expansion resembles derivative feedback; exponential temporal laws permit auxiliary filter states. These are useful analytical viewpoints on the same dynamics. They capture local behavior or chosen realizations, while the nonlinear traveling response requires the full imprint beyond the equilibrium displacement. We make the reductions explicit to identify which geometric information each one retains.

\subsection{Distributed-delay and Volterra representation}
With fixed $w$ and $H=-\nabla^2\Phi(0)$, linearization at a constant history gives
\[
 \dot y=m*\dot y+\kappa H\int_0^\infty w(s)
 [y(t)-y(t-s)]\dd s.
\]
The tail-kernel identity
\begin{equation}\label{eq:tail}
 \int_0^\infty w(s)[y(t)-y(t-s)]\dd s
 =\int_0^\infty q(r)\dot y(t-r)\dd r
\end{equation}
recovers a conventional linear Volterra convolution. Its characteristic equation is \eqref{eq:char}; the positive-kernel threshold in Proposition~\ref{prop:spec} belongs to established Volterra theory \cite{Gripenberg1990,Naito2007,Ngoc2008}. It is the \emph{tangent limit} of our model, not the source of the geometric separation theorem.

In Laplace notation, the linear memory difference has multiplier
\begin{equation}\label{eq:transfer}
 W_0-\widehat w(p)
 =p\widehat q(p),\qquad
 W_0=\int_0^\infty w(s)\dd s.
\end{equation}
If $M_1<\infty$, then $W_0-\widehat w(p)=M_1p+o(p)$ as $p\to0$ in the transform domain. Hence the linearized feedback resembles filtered derivative control at low frequencies. For $w(s)=\alpha e^{-\alpha s}$ it becomes $p/(p+\alpha)$; for the two-stage field kernel $w(s)=A\alpha se^{-\alpha s}$ it is
\begin{equation}\label{eq:double}
 W_0-\widehat w(p)=\frac{Ap(p+2\alpha)}{\alpha(p+\alpha)^2}.
\end{equation}
These transfer functions cannot distinguish force-germ equivalent imprints in Theorem~\ref{thm:germ}.

\subsection{Fading-memory phase spaces and state dependence}
In infinite-delay functional differential equations, the full history is itself the state \cite{HaleKato1978,Kolmanovskii1999}. The fading-memory tradition \cite{ColemanMizel1966,ColemanMizel1968} formalizes how increasingly remote history may have a decreasing influence on the present state. Our fixed-$w$ force is a nonlinear functional of that history, and the adaptive law \eqref{eq:adaptive} is a state-dependent functional weighting. This is related to, but not identical with, state-dependent \emph{evaluation delays} of the form $x(t-\tau(x_t))$ studied in \cite{Hartung2006}. The distinction matters: changing the temporal weight is not necessarily a change in the time argument itself. Well-posedness and differentiability must be proved in the corresponding phase space, and the absence of a derivative term in Proposition~\ref{prop:blind} is a feature of the stationary imprint-gradient factor, not a universal fact about state-dependent delays.

\subsection{Markovian realization for exponential memories}
For exponential temporal kernels a part of the apparently non-Markovian structure can be realized by additional internal states. Take $m(s)=\rho\mu e^{-\mu s}$ and $w(s)=A\alpha se^{-\alpha s}$, and define for every spatial point $z$
\begin{align}
 U(t)&=\rho\mu\int_0^\infty e^{-\mu s}\dot x(t-s)\dd s,\notag\\
 h(z,t)&=A\alpha\int_0^\infty e^{-\alpha s}\Phi(z-x(t-s))\dd s,\notag\\
 S(z,t)&=A\alpha\int_0^\infty se^{-\alpha s}\Phi(z-x(t-s))\dd s.
\end{align}
Under the usual differentiation and decay assumptions,
\begin{align}
 \dot U&=-\mu U+\rho\mu\dot x,\\[-1pt]
 \partial_t h&=-\alpha h+A\alpha\Phi(z-x(t)),\\[-1pt]
 \partial_t S&=-\alpha S+h,\\[-1pt]
 \dot x&=U-\kappa\nabla S(x(t),t).
\end{align}
This augmented realization has finite-order temporal filters and field states indexed by $z\in\R^n$. The state-space geometry persists in the coupling $\nabla S(x(t),t)$.

\subsection{Lyapunov and small-gain viewpoints}
Classical Lyapunov--Krasovskii analysis aims to establish decay through a suitable functional of the complete history \cite{Kolmanovskii1999}. Theorem~\ref{thm:arrest} instead uses a direct weighted comparison on the exact geometric velocity factorization. These complementary estimates have different strengths: the global factor $L_A$ gives a uniform certificate, while trajectory-specific estimates retain cancellations along prescribed geometric histories. Theorem~\ref{thm:germ} gives a complementary negative result: no local Lyapunov or spectral argument restricted to a neighborhood where the two forces coincide can by itself determine the presence or absence of distant translating branches.

A concrete control interpretation is obtained by applying a geometric memory input to $\dot x=Ax+Bu$. For example,
\[
 u=-Kx+\gamma B^\top\nabla S_x(x,t)
\]
produces the closed loop $\dot x=(A-BK)x+\gamma BB^\top\nabla S_x(x,t)$. Its equilibrium variation recovers a distributed-derivative compensator, while the full nonlinear law depends on the geometry of visited states. Stabilization of a specified target then requires analysis beyond the translation-invariant pure-memory model \eqref{eq:main}. The resulting class connects geometric memory to state-feedback design and to distributed-derivative compensation near a chosen equilibrium.

\section{Geometric classification and structural consequences}
The results can be read as a hierarchy of increasingly informative observations. An equilibrium experiment identifies a local response. A constant-speed experiment probes a weighted range of spatial displacements. A perturbation of that moving experiment reveals derivatives of the force at those nonzero displacements. The constructive separation theorem shows why these additional levels are mathematically consequential: a complete local equilibrium description can coexist with globally different velocity-arrest properties and moving-orbit families.

The equilibrium linearization retains only the local Hessian of the spatial imprint and the temporal transform. The exact geometric response $F_a(r)$ samples the complete imprint along nonzero translating histories. Theorem~\ref{thm:geostable} exploits this structural distinction inside a single hereditary operator class: exact local nonlinear force identity is compatible with global velocity arrest for one imprint and coexistence of attracting rest and motion for another. Its defining structural content is the \emph{simultaneous} preservation of the entire local force germ, positive temporal weighting, monotone radial imprint geometry, and a common feedback gain, together with a proven difference in nonlinear asymptotic motion. The contraction estimates and positive-exponent test draw on Volterra analysis; the geometric construction connects these tools in an operator-level local-to-global non-identifiability result. The positive longitudinal exponent and nonlinear attraction estimate concern the full infinite-history equation, independently of any scalar speed-drift approximation. The construction applies within the specified smooth, non-analytic imprint class and compactly supported temporal weights; other kernel classes are addressed by the general spectral and arrest statements.

The preceding analysis yields three complementary layers of stability information. The equilibrium tangent equation determines the local spectral boundary through $a(0)$ and the first moment of temporal memory. The exact traveling response $F_a(r)$ selects complete moving solutions. Its derivative $F'_a(r)$ and its longitudinal force factor $b_r$ then control growth and attraction of these moving branches. These three layers sample progressively richer information about the spatial imprint.

Theorem~\ref{thm:geostable} combines them into a concrete dynamical separation. At a common subcritical gain, two systems have identical complete nonlinear forces on an open neighborhood of rest in the \emph{same} velocity-history phase space. The first has globally arrested velocities for every admissible bounded prehistory. In the second, locally attracting rest coexists with a spectrally unstable lower-speed branch and a locally exponentially attracting upper-speed branch. Theorem~\ref{thm:flat} describes the infinite-memory counterpart: a difference exponentially small with respect to the velocity-history amplitude, with all derivatives at equilibrium equal, can still change the moving-orbit structure.

This formulation also distinguishes geometrically induced speed stability from symmetry. In one dimension the upper moving branch can attract in velocity and therefore approach a straight traveling orbit up to an asymptotic position offset. In dimension $n\ge2$, the continuum of velocity directions supplies neutral rotational modes, and the correct notion of transverse orbital stability is formulated on a Euclidean-symmetry quotient. Proposition~\ref{prop:transverse} establishes the corresponding transverse spectral exclusion for positive kernels.

The fixed-weight radial theory serves as a foundation for three extensions: anisotropic imprints with direction-dependent traveling response surfaces; adaptive weights in which speed branches solve genuinely nonlinear history-dependent balance laws; and geometric invariant-set stability for rotating or recurrent motions. In each case the temporal kernel, geometric derivative, and phase-space topology are explicit structural inputs.

\section{Conclusions}
A system with self-generated state-space memory selects its past in two ways: by the \emph{age} of each imprint and by its \emph{spatial relation} to the current state. Their coupling produces a nonlinear hereditary force whose effective velocity kernel changes with the trajectory. The exact geometry-to-velocity factorization makes this coupling explicit and yields global arrest estimates, traveling-speed balance equations, and moving-branch stability tests in a shared analytical setting. Linear Volterra equations and distributed-derivative control laws appear as tangent or specialized descriptions of this operator.

The principal result is a constructive form of \emph{geometric non-identifiability}. Two smooth, radially decreasing memory imprints, equipped with identical positive temporal weights and one common feedback gain, exert exactly the same nonlinear force throughout an open neighborhood of rest in a common velocity-history phase space. Yet the first produces exponential velocity arrest for every bounded admissible prehistory, while the second admits coexistence of locally attracting rest and locally attracting uniform translation, together with a moving branch possessing a positive longitudinal characteristic exponent. The distinction is established for the original hereditary equations. It is not inferred from a surrogate autonomous speed equation or from equilibrium spectral data.

The exponential-flatness result further quantifies the loss of information in equilibrium expansions: for exponentially decaying infinite memory, two forces can differ in every velocity neighborhood of rest while agreeing to all finite orders at the zero history. These results locate the essential geometric data in the full spatial imprint and its response along nonzero trajectory increments. The normalized velocity-memory case $M_0=1$ exhibits an exact obstruction to nonzero uniform translations under strictly repulsive radial imprints; subnormalized memory $M_0<1$ permits the force balance behind the constructed coexistence phenomenon.

The ensuing mathematical program concerns the global organization of basins and invariant sets, anisotropic imprint geometry, and history-adaptive temporal weights. The present theorems provide exact conditions for the identified arrest and moving branches, together with an explicit demonstration of the global dynamics that local memory-force information cannot determine.

\section*{Declarations}
\noindent\textbf{Data availability:} The figures are generated from the stated one-dimensional quadratures.\\
\textbf{Conflict of interest:} The author declares no conflict of interest.

\end{document}